\documentclass[12pt]{amsart}
\usepackage[latin1]{inputenc}
\usepackage{mathptmx}
\usepackage{amscd}
\usepackage{color}

\newtheorem{theorem}{Theorem}[section]

\newtheorem{corollary}[theorem]{Corollary}
\newtheorem{proposition}[theorem]{Proposition}

\theoremstyle{definition}

\let\frak=\mathfrak
\let\Bbb=\mathbb
\let\phi=\varphi

\def\Hom{\operatorname{Hom}}
\def\Ext{\operatorname{Ext}}

\def\Spec{\operatorname{Spec}}

\def\Ker{\operatorname{Ker}}

\def\Im{\operatorname{Im}}

\let\oldbigwedge\bigwedge
\def\BIGwedge{{\textstyle\oldbigwedge}}
\def\medwedge{{\scriptstyle\oldbigwedge}}
\def\bigwedge{\mathchoice{\BIGwedge}{\BIGwedge}{\medwedge}{}}

 \DeclareMathOperator{\Supp}{Supp}

\let\epsilon=\varepsilon
\let\tilde=\widetilde

\begin{document}

\title[Bass numbers of local cohomology modules...]{Bass numbers of local cohomology modules with respect a pair of ideals}
\author{  Sh. Payrovi }
\author{ M.  Lotfi Parsa }
\author{ S.  Babaei }
\address{I. K. International University,
 Postal Code: 34149-1-6818
Qazvin - IRAN } \email{shpayrovi@ikiu.ac.ir} \email{
lotfi.parsa@yahoo.com}\email{
 sakine-babaei@yahoo.com}

\begin{abstract}
Let $R$ be a  Noetherian local ring, $I$ and $J$ two ideals of $R$, $M$
an $R$-module and  $s$ and $t$ two integers. We
study the relationship between the Bass numbers of $M$ and
$H^{i}_{I,J}(M)$. We show that
 $\mu^t(M)\leq\sum_{i=0}^{t}\mu^{t-i}(
H^{i}_{I,J}(M))$ and  $\mu^s(H^{t}_{I,J}(M))\leq
\sum_{i=0}^{t-1}\mu^{s+t+1-i}(H^{i}_{I,J}(M))+\mu^{s+t}(M)+\sum_{i=t+1}^{s+t-1}\mu^{s+t-1-i}(H^{i}_{I,J}(M))$.
 As a consequence, it follows  that
if $I$ is a principal ideal of $R$ and $M$ is a minimax $R$-module,
then  $\mu^j(H^{i}_{I,J}(M))$
is finite for all $i\in\Bbb N_0$ and all $j\in\Bbb N_0$.
\end{abstract}

{\footnote{$\hspace*{-5mm}$ 2010 Mathematics Subject
Classification. 13D45, 13D07. }
\vspace{-0.5cm} \maketitle
\section{Introduction }

Throughout this paper, $R$ is a commutative Noetherian ring with
non-zero identity, $I$ and $J$ are two ideals of $R$, $M$ is an
$R$-module and $s$ and $t$ are two integers. For notations and
terminologies not given in this paper, the reader is referred to
\cite{bs}, \cite{bh} and \cite{ty} if necessary.

The theory of local cohomology, which was introduced by Grothendieck \cite{g}, is a useful
tool for attacking problems in commutative algebra and algebraic geometry.
Bijan-Zadeh  \cite{b} introduced  the local cohomology
modules with respect to a system of ideals, which is a generalization of ordinary local cohomology modules.
As a special case of these extend modules,
Takahashi, Yoshino and Yoshizawa \cite{ty} defined the local cohomology
modules with respect to a pair of ideals. To be
more precise, let ${\rm W}(I, J)=\{\frak p\in \Spec(R):
I^t\subseteq J+\frak p ~~~~\text{for some positive
integer}~~~~t\}$. The set of elements $x$ of $M$ such that
$\Supp_RRx\subseteq {\rm W}(I, J)$, is said to be $(I, J)$-torsion
submodule of $M$ and is denoted by $\Gamma_{I,J}(M)$.
 $\Gamma_{I,J}(-)$ is a covariant,
$R$-linear functor from the category of $R$-modules to itself. For
an integer $i$, the local cohomology functor $H^{i}_{I,J}(-)$ with
respect to $(I, J)$, is defined to be the $i$-th right derived
functor of $\Gamma_{I,J}(-)$. Also $H^{i}_{I,J}(M)$ is called the
$i$-th local cohomology module of $M$ with respect to $(I, J)$. If
$J=0$, then $H^{i}_{I,J}(-)$ coincides with the ordinary local
cohomology functor $H^{i}_{I}(-)$.  Let
$\tilde{{\rm W}}(I, J)=\{\frak a\unlhd R: I^t\subseteq J+\frak
a~~~~\text{for some positive integer} ~~~~t\}$.
 It is easy to see that
$$\Gamma_{I,J}(M)=\{x\in M: \exists\ \frak a\in {\rm \tilde W}(I,J),\ \frak ax=0\}
=\bigcup_{\frak a\in {\rm \tilde W}(I,J)}(0:_{M}\frak a).$$


An important problem in commutative Algebra is to determine when
the Bass numbers of the $i$-th local cohomology module is finite. In
\cite{h} Huneke conjectured that if $(R,\frak m, k)$ is a regular
local ring, then for any prime ideal $\frak p$ of $R$ the Bass
numbers $\mu^j(\frak p, H^{i}_{I}(R)) = \dim_{k(p)}
\Ext^j_{R_{\frak p}} (k(\frak p), H^{i}_{IR_{\frak p}}(R_{\frak
p}))$ are finite for all $i\in\Bbb N_0$ and all $j\in\Bbb N_0$.
There are some evidences
that this conjecture is true;  see \cite{hs}, \cite{l1} and
\cite{l2}. On the other hand, there is a negative answer to the conjecture
(over a non-regular ring) that is due to Hartshorne, see
\cite{ha}.
However the conjecture does not hold over a non-regular ring,
Kawasaki \cite{k} proved that if  $R$ is a local ring, $I$ is a principal ideal of $R$, and
$M$ is a finitely generated $R$-module,
then  $\mu^j(\frak p, H^{i}_{I}(M))$
is finite for all $\frak p\in\Spec(R)$, all $i\in\Bbb N_0$ and all $j\in\Bbb N_0$.

 Dibaei and Yassemi \cite{dy} studied the
relationship between the Bass numbers of an $R$-module and those of its local
cohomology modules. They show that if $R$ is a local
ring, then
 $$\mu^t(M)\leq\sum_{i=0}^{t}\mu^{t-i}(
H^{i}_{I}(M))$$ and  $$\mu^s(H^{t}_{I}(M))\leq
\sum_{i=0}^{t-1}\mu^{s+t+1-i}(H^{i}_{I}(M))+\mu^{s+t}(M)+\sum_{i=t+1}^{s+t-1}\mu^{s+t-1-i}(H^{i}_{I}(M)).$$
In \ref{13}, by a different method, we generalize this result for local cohomology modules
with respect to a pair of ideals. As a consequence,  it follows  that
if $R$ is a local ring and $I$ is a principal ideal of $R$, then  $\mu^j(H^{i}_{I,J}(M))$
is finite for all $i\in\Bbb N_0$ and all $j\in\Bbb N_0$, where
$M$ is a minimax $R$-module; see \ref{123}.

In  section 3, we get some isomorphisms  about the
extension functors  of local cohomology modules, which imply
 some equalities about the
 Bass numbers of local cohomology modules.
\section{Bass numbers }
Recall that $R$ is a Noetherian ring, $I$ and $J$ are ideals of
$R$ and $M$ is an $R$-module.
 The following theorem is the main
result of this paper.

\begin{theorem}\label{11}
Let $N$ be an  $(I,J)$-torsion $R$-module. Then
\begin{itemize}
\item[(i)] $$\dim_R \Ext_R^t(N, M)\leq\sum_{i=0}^{t}\dim_R
\Ext_R^{t-i}(N, H^{i}_{I,J}(M)).$$

\item[(ii)]
\begin{eqnarray*}
\dim_R \Ext_R^{s}(N, H^{t}_{I,J}(M))&\leq& \sum_{i=0}^{t-1}\dim_R
\Ext_R^{s+t+1-i}(N, H^{i}_{I,J}(M))\\&+&\dim_R \Ext_R^{s+t}(N,
M)\\&+&\sum_{i=t+1}^{s+t-1}\dim_R \Ext_R^{s+t-1-i}(N,
H^{i}_{I,J}(M)).
\end{eqnarray*}
\end{itemize}

\end{theorem}
\begin{proof}
Let $F(-)=\Hom_R(N,-)$ and $G(-)=\Gamma_{I,J}(-)$. Then we have
$FG(M)= \Hom_R(N, M)$. By \cite[Theorem 11.38]{r}, there is
the Grothendieck spectral sequence
$$E_2^{p,q}:=\Ext_R^{p}(N, H^{q}_{I,J}(M))\Rightarrow \Ext_R^{p+q}(N,
M).$$  There is a finite filtration
$$0=\phi^{p+q+1}H^{p+q}\subseteq\phi^{p+q}H^{p+q}\subseteq\cdots
\subseteq\phi^{1}H^{p+q}\subseteq\phi^{0}H^{p+q}=\Ext_R^{p+q}(N,M)$$
such that $E_\infty^{p+q-i,i}\cong
\phi^{p+q-i}H^{p+q}/\phi^{p+q+1-i}H^{p+q}$ for all $i\leq p+q$.

(i) We have to show that $\dim_R
\phi^{0}H^t\leq\sum_{i=0}^{t}\dim_R E_2^{t-i,i}$.  The
sequence
$$0\longrightarrow\phi^{t+1-i}H^t\longrightarrow\phi^{t-i}H^t\longrightarrow
E_\infty^{t-i,i}\longrightarrow 0$$ is exact for all $i\leq t$. It follows that
\begin{eqnarray*}
\dim_R \phi^{0}H^t&\leq&\dim_R \phi^{1}H^t+\dim_R E_\infty^{0,t}
\\&\leq&\dim_R \phi^{2}H^t+\dim_R E_\infty^{1,t-1}+\dim_R E_\infty^{0,t}
\\&\leq&\cdots\\&\leq&\sum_{i=0}^{t}\dim_R
E_\infty^{t-i,i}.
\end{eqnarray*}
Since $E_\infty^{t-i,i}$ is a subquotient of $E_2^{t-i,i}$ for all
$i\leq t$, thus $\dim_R E_\infty^{t-i,i}\leq \dim_R
E_2^{t-i,i}$ and  the claim holds.

(ii) We have to show that
$$\dim_R E_{2}^{s,t}\leq \sum_{i=0}^{t-1}\dim_R
E_{2}^{s+t+1-i,i}+\dim_R
\phi^{0}H^{s+t}+\sum_{i=t+1}^{s+t-1}\dim_R E_{2}^{s+t-1-i,i}.$$
The  sequences
$$0\longrightarrow \Ker d_{t+1-i}^{s,t}\longrightarrow E_{t+1-i}^{s,t}\stackrel{d_{t+1-i}^{s,t}}\longrightarrow
E_{t+1-i}^{s+t+1-i,i}$$ and
$$0\longrightarrow \Im d_{t+1-i}^{s-t-1+i,2t-i}\longrightarrow
\Ker d_{t+1-i}^{s,t}\longrightarrow
E_{t+2-i}^{s,t}\longrightarrow0$$ are exact
for any integer $i$. It follows that
\begin{eqnarray*}
\dim_R E_{2}^{s,t}&\leq&\dim_R E_{2}^{s+2,t-1}+\dim_R \Ker
d_{2}^{s,t}
\\&\leq&\dim_R E_{2}^{s+2,t-1}+\dim_R
E_{3}^{s,t}+\dim_R \Im d_{2}^{s-2,t+1}
\\&\leq&\dim_R E_{2}^{s+2,t-1}+\dim_R
E_{3}^{s+3,t-2}+\dim_R \Ker d_{3}^{s,t}+\dim_R \Im
d_{2}^{s-2,t+1}\\&\leq&\dim_R E_{2}^{s+2,t-1}+\dim_R
E_{3}^{s+3,t-2}+\dim_R E_{4}^{s,t}+\dim_R \Im
d_{3}^{s-3,t+2}\\&+&\dim_R \Im d_{2}^{s-2,t+1}
\\&\leq&\cdots\\&\leq&\sum_{i=0}^{t-1}\dim_R E_{t+1-i}^{s+t+1-i,i}+\dim_R
E_{s+t+2}^{s,t}+\sum_{i=t+1}^{s+t-1}\dim_R \Im
d_{1-t+i}^{s+t-1-i,i}.
\end{eqnarray*}
Since $E_{t+1-i}^{s+t+1-i,i}$ is a subquotient of
$E_{2}^{s+t+1-i,i}$, $E_{s+t+2}^{s,t}=E_\infty^{s,t}$ is a
subquotient of $\phi^{0}H^{s+t}$, and $\Im d_{1-t+i}^{s+t-1-i,i}$
is a subquotient of $E_{2}^{s+t-1-i,i}$, the claim
follows.
 \end{proof}

\begin{corollary}\label{12}
Let $N$ be an  $(I,J)$-torsion $R$-module.
Let $\Ext_R^{s+t-i}(N,
H^{i}_{I,J}(M))=0$ for  all $i\neq t$ with $i\leq s+t$,
$\Ext_R^{s+t+1-i}(N, H^{i}_{I,J}(M))=0$ for all $i<t$, and let
$\Ext_R^{s+t-1-i}(N, H^{i}_{I,J}(M))=0$ for all $t<i<s+t$. Then
$$\dim_R\Ext_R^{s}(N, H^{t}_{I,J}(M))=\dim_R\Ext_R^{s+t}(N, M).$$
\end{corollary}

 \begin{corollary}\label{15}
Suppose that $N$ is a finitely generated $\frak{a}$-torsion
$R$-module for some $\frak{a}\in \tilde{{\rm W}}(I,J)$. Then
 \begin{itemize}
\item[(i)] $$\dim_R H_\frak{a}^t(N, M)\leq\sum_{i=0}^{t}\dim_R
\Ext_R^{t-i}(N, H^{i}_{I,J}(M)).$$

\item[(ii)]
 \begin{eqnarray*} \dim_R \Ext_R^{s}(N,
H^{t}_{I,J}(M))&\leq& \sum_{i=0}^{t-1}\dim_R \Ext_R^{s+t+1-i}(N,
H^{i}_{I,J}(M))\\&+&\dim_R
H^{s+t}_{\frak{a}}(N,M)\\&+&\sum_{i=t+1}^{s+t-1}\dim_R
\Ext_R^{s+t-1-i}(N, H^{i}_{I,J}(M)).
\end{eqnarray*}
\end{itemize}
\end{corollary}
\begin{proof}
Note that $\Gamma_\frak{a}(N)\subseteq \Gamma_{I,J}(N)$, and
by \cite[Lemma 2.1]{dst} we have $\Ext_R^{i}(N, M)\cong
H^i_\frak{a}(N, M)$ for any integer $i$.
\end{proof}

When $(R, \frak m)$ is a local ring, we put $\mu^i(M):=\mu^i(\frak
m, M)$. The following result is a generalization of the main
results of \cite{dy}.

 \begin{corollary}\label{13}
If $(R, \frak m)$ is a local ring, then
 \begin{itemize}
\item[(i)] $$\mu^t(M)\leq\sum_{i=0}^{t}\mu^{t-i}(
H^{i}_{I,J}(M)).$$

\item[(ii)] $$\mu^s(H^{t}_{I,J}(M))\leq
\sum_{i=0}^{t-1}\mu^{s+t+1-i}(H^{i}_{I,J}(M))+\mu^{s+t}(M)+\sum_{i=t+1}^{s+t-1}\mu^{s+t-1-i}(H^{i}_{I,J}(M)).$$
\end{itemize}
\end{corollary}
\begin{proof}
In \ref{11}, put $N=R/{\frak m}$.
\end{proof}

\begin{corollary}\label{14}
Let $(R, \frak m)$ be a local ring. Let
$\mu^{s+t-i}(H^{i}_{I,J}(M))=0$ for  all $i\neq t$ with $i\leq
s+t$, $\mu^{s+t+1-i}(H^{i}_{I,J}(M))=0$ for all $i<t$, and
$\mu^{s+t-1-i}(H^{i}_{I,J}(M))=0$ for all $t<i<s+t$. Then
$\mu^{s}(H^{t}_{I,J}(M))=\mu^{s+t}(M)$.
\end{corollary}


\begin{corollary}\label{0001}
Let $(R, \frak m)$ be a local ring and
$I=(a_1,a_2,\ldots,a_t)$. Then
$$\mu^s(H^{t}_{I,J}(M))\leq
\sum_{i=0}^{t-1}\mu^{s+t+1-i}(H^{i}_{I,J}(M))+\mu^{s+t}(M).$$
\end{corollary}
\begin{proof}
The claim follows by \ref{13}(ii) and \cite[Proposition 4.11]{ty}.
\end{proof}

\begin{corollary}\label{123}
Let $(R, \frak m)$ be a local ring and $I$  a principal ideal of $R$. Let
$M$ be a minimax $R$-module. Then $\mu^j(H^{i}_{I,J}(M))$
is finite for all $i\in\Bbb N_0$ and all $j\in\Bbb N_0$.
\end{corollary}
\begin{proof}
Since $I$ is principal, it follows by \cite[Proposition 4.11]{ty} that
$H^{i}_{I,J}(M)=0$ for all $i>1$. Therefore
$\mu^s(H^{1}_{I,J}(M))\leq
\mu^{s+2}(H^{0}_{I,J}(M))+\mu^{s+1}(M)$, by \ref{0001}.
Now the claim follows by this fact that any minimax module has finite Bass numbers.
\end{proof}

\begin{proposition}\label{17}
Let $N$ be an  $(I,J)$-torsion $R$-module. Then the following are
true for  all $\frak p\in \Spec(R)$ and all $j\in\Bbb N_0$:
 \begin{itemize}
\item[(i)] $$\mu^j(\frak p, \Ext_R^t(N,
M))\leq\sum_{i=0}^{t}\mu^j(\frak p, \Ext_R^{t-i}(N,
H^{i}_{I,J}(M))).$$

\item[(ii)]
\begin{eqnarray*} \mu^j(\frak p, \Ext_R^{s}(N,
H^{t}_{I,J}(M)))&\leq& \sum_{i=0}^{t-1}\mu^j(\frak p,
\Ext_R^{s+t+1-i}(N,
H^{i}_{I,J}(M)))\\
&+&\mu^j(\frak p, \Ext_R^{s+t}(N, M))\\
&+&\sum_{i=t+1}^{s+t-1}\mu^j(\frak p, \Ext_R^{s+t-1-i}(N,
H^{i}_{I,J}(M))).
\end{eqnarray*}
\end{itemize}
\end{proposition}
\begin{proof}
The proof is similar to that of \ref{11}.
\end{proof}

\section{Some isomorphisms }

In this section, we get some isomorphisms and equalities about the
extension functors and the Bass numbers of local cohomology modules, respectively. The following result is a
generalization of \cite[Theorem 3.5]{atv}.

\begin{theorem}\label{25}
Let $N$ be an  $(I,J)$-torsion $R$-module. Let $\Ext_R^{s+t-i}(N,
H^{i}_{I,J}(M))=0$ for all $i\neq t$ with $i\leq s+t$,
$\Ext_R^{s+t+1-i}(N, H^{i}_{I,J}(M))=0$ for all $i<t$,  and let
$\Ext_R^{s+t-1-i}(N, H^{i}_{I,J}(M))=0$ for all $t<i<s+t$. Then
$$\Ext_R^{s}(N, H^{t}_{I,J}(M))\cong\Ext_R^{s+t}(N, M).$$
\end{theorem}
\begin{proof}
Let
$F(-)=\Hom_R(N,-)$ and $G(-)=\Gamma_{I,J}(-)$. Then we have
 $FG(M)=
\Hom_R(N, M)$. By \cite[Theorem 11.38]{r}, there is the
Grothendieck  spectral sequence
$$E_2^{p,q}:=\Ext_R^{p}(N, H^{q}_{I,J}(M))\Rightarrow \Ext_R^{p+q}(N,
M).$$  There is a finite filtration
$$0=\phi^{t+1}H^{t}\subseteq\phi^{t}H^{t}\subseteq\cdots
\subseteq\phi^{1}H^{t}\subseteq\phi^{0}H^{t}=\Ext_R^{t}(N,M)$$
such that $E_\infty^{t-i,i}\cong
\phi^{t-i}H^{t}/\phi^{t+1-i}H^{t}$ for all $i\leq t$.
We have to show that $\phi^0H^{s+t}\cong E_2^{s,t}$. Our
hypothesis imply that $E_2^{s+t-i,i}=0$ for all $i\neq t$ with
$i\leq s+t$. So $E_{\infty}^{s+t-i,i}=0$ for all $i\neq t$ with
$i\leq s+t$.  The  sequence
$$0\longrightarrow\phi^{s+t+1-i}H^{s+t}\longrightarrow\phi^{s+t-i}H^{s+t}\longrightarrow
E_\infty^{s+t-i,i}\longrightarrow 0$$ is exact for any integer $i$. It follows that
$\phi^{s}H^{s+t}\cong E_{\infty}^{s,t}$ and $\phi^{s}H^{s+t}\cong
\phi^{0}H^{s+t}$, and so that $\phi^{0}H^{s+t}\cong
E_{\infty}^{s,t}$. Therefore it is enough to show that
$E_{\infty}^{s,t}\cong E_2^{s,t}$. Our hypothesis imply that
$E_{t+1-i}^{s+t+1-i,i}=0$ for all $i<t$, and
$E_{1-t+i}^{s+t-1-i,i}=0$ for all $t<i<s+t$. So
$E_{t+1-i}^{s-t-1+i,2t-i}=0$ for all $t-s<i<t$. Note that if
$i\leq t-s$, then $E_{t+1-i}^{s-t-1+i,2t-i}=0$. Therefore
$E_{t+1-i}^{s-t-1+i,2t-i}=0$ for all $i<t$, and so that $\Im
d_{t+1-i}^{s-t-1+i,2t-i}=0$ for all $i<t$. The
 sequences
$$0\longrightarrow \Ker d_{t+1-i}^{s,t}\longrightarrow E_{t+1-i}^{s,t}\stackrel{d_{t+1-i}^{s,t}}\longrightarrow
E_{t+1-i}^{s+t+1-i,i}$$ and
$$0\longrightarrow \Im d_{t+1-i}^{s-t-1+i,2t-i}\longrightarrow
\Ker d_{t+1-i}^{s,t}\longrightarrow
E_{t+2-i}^{s,t}\longrightarrow0$$
are exact for any integer $i$. It follows that  $E_2^{s,t}\cong
E_{s+t+2}^{s,t}=E_{\infty}^{s,t}$, and the claim follows.
\end{proof}

\begin{corollary}\label{26}
Let $\frak{p}\in {\rm W}(I,J)$. Let
$\Ext_R^{s+t-i}(R/\frak{p}, H^{i}_{I,J}(M))=0$ for all $i\neq t$ with
 $i\leq s+t$, $\Ext_R^{s+t+1-i}(R/\frak{p},
H^{i}_{I,J}(M))=0$ for all $i<t$, and
$\Ext_R^{s+t-1-i}(R/\frak{p}, H^{i}_{I,J}(M))=0$ for all
$t<i<s+t$. Then $\mu^{s}(\frak{p},
H^{t}_{I,J}(M))=\mu^{s+t}(\frak{p}, M)$.
\end{corollary}
\begin{proof}
We note that
$\Ext_{R_\frak{p}}^{s}(R_\frak{p}/{\frak{p}R_\frak{p}},
H^{t}_{I,J}(M)_\frak{p})\cong\Ext_{R_\frak{p}}^{s+t}(R_\frak{p}/{\frak{p}R_\frak{p}},
M_\frak{p})$, by \ref{25}.
\end{proof}

\begin{corollary}\label{27}
Suppose that $N$ is a finitely generated $\frak{a}$-torsion
$R$-module for some $\frak{a}\in \tilde{{\rm W}}(I,J)$. Suppose that
 $\Ext_R^{s+t-i}(N, H^{i}_{I,J}(M))=0$ for all $i\neq t$ with
 $i\leq s+t$, $\Ext_R^{s+t+1-i}(N, H^{i}_{I,J}(M))=0$ for all
$i<t$,  and $\Ext_R^{s+t-1-i}(N, H^{i}_{I,J}(M))=0$ for all
$t<i<s+t$.
 Then $\Ext_R^{s}(N, H^{t}_{I,J}(M))\cong\Ext_R^{s+t}(N, M)\cong H^{s+t}_{\frak{a}}(N,M)$.
\end{corollary}
\begin{proof}
The result follows by  \ref{25} and  \cite[Lemma 2.1]{dst}.
\end{proof}

\begin{corollary}\label{28}
Suppose that $N$ is a finitely generated $\frak{a}$-torsion
$R$-module for some $\frak{a}\in \tilde{{\rm W}}(I,J)$. Suppose that
 $\Ext_R^{j-i}(N, H^{i}_{I,J}(M))=0$ for $j=t, t+1$ and all
$i<t$. Then $\Hom_R(N, H^{t}_{I,J}(M))\cong\Ext_R^t(N,
M)\cong H^{t}_{\frak{a}}(N,M)$.
\end{corollary}
\begin{proof}
In \ref{27}, put $s=0$.
\end{proof}

\begin{corollary}\label{29}
Let $N$ be a finitely generated  $R$-module with $\Supp_R
N={\rm V}(\frak a)$ for some $\frak{a}\in \tilde{{\rm W}}(I,J)$.
If $\Ext_R^j(N, H^{i}_{I,J}(M))=0$ for all $i<t$ and all $j\leq t+1-i$, then $\Hom_R(N,
H^{t}_{I,J}(M))\cong\Ext_R^{t}(N, M)\cong H^{t}_{\frak{a}}(N,M)\cong\Hom_R(N,
H^{t}_{\frak{a}}(M))$.
\end{corollary}
\begin{proof}
The result follows by \ref{28} and \cite[Corollary 2.5]{pl}.
\end{proof}


\end{document}